\documentclass[12pt,reqno]{amsart}
\usepackage{amsfonts}
\usepackage[a4paper,inner=1in,outer=1in,vmargin=1in,marginparwidth=1in]{geometry}
\usepackage{amsmath,amsthm,amsfonts,amssymb}
\usepackage{bbm}
\usepackage{cite}
\usepackage{color}
\usepackage{enumerate}

\makeatletter
\@namedef{subjclassname@2020}{\textup{2020} Mathematics Subject Classification}
\makeatother
\def\ad{\text{ad}}
\def\C{\mathbb{C}}

\def\Z{\mathbb{Z}}

\newcommand{\op}{\oplus}
\newcommand{\mf}{\mathfrak}
\newcommand{\dis}{\displaystyle}
\newcommand{\bgop}{\bigoplus}
\newcommand{\ot}{\otimes}
\newcommand{\al}{\alpha}

\newcommand{\ga}{\gamma}
\newcommand{\wtil}{\widetilde}

\newcommand{\la}{\lambda}

\newcommand{\de}{\delta}
\newcommand{\what}{\widehat}

\newtheorem{theorem}{Theorem}[section]
\newtheorem{definition}[theorem]{Definition}
\newtheorem{lemma}[theorem]{Lemma}

\newtheorem{corollary}[theorem]{Corollary}
\newtheorem{proposition}[theorem]{Proposition}

\numberwithin{equation}{section}
\allowdisplaybreaks
\begin{document}
	\title{Derivations and Biderivations of affine-Virasoro Lie algebras }

	\author[]{Priyanshu Chakraborty, Yufeng Yao, Kaiming Zhao}
	\address{Priyanshu Chakraborty: School of Mathematical Sciences, Ministry of Education Key Laboratory of Mathematics and Engineering Applications and Shanghai Key Laboratory of PMMP,
		East China Normal University, No. 500 Dongchuan Rd., Shanghai 200241, China.}
	\email{priyanshu@math.ecnu.edu.cn, priyanshuc437@gmail.com}
	\address{Yufeng Yao: Department of Mathematics
		Shanghai Maritime University
		No. 1550 Haigang Ave., Shanghai 201306, China}\email{yfyao@shmtu.edu.cn}
	\address{ Kaiming Zhao: Department of Mathematics, Wilfrid Laurier University, Waterloo, ON, Canada N2L3C5}
	\email{kzhao@wlu.ca}
	
	\date{}
	
	\maketitle
	\begin{abstract}
		In this paper we determine all derivations and biderivations of an affine-Virasoro Lie algebra associated with a finite-dimensional complex simple Lie algebra $\mf g$. We prove that all the derivations and biderivations of affine-Virasoro Lie algebras are inner.
	\end{abstract}
	
	\section{Introduction}

Derivations and biderivations play a crucial role in understanding the structure of various algebras, as explored in articles \cite{B,MB,MB1,BMM,BZ}. Notably, Bresar et al. introduced the concept of biderivations for rings in \cite{BMM}, demonstrating that all biderivations of non-commutative prime rings are inner. Subsequently, biderivations for several Lie algebras were completely  determined. In particular, in \cite{WYC} it was shown that skew-symmetric biderivations of finite-dimensional Lie algebras are inner. In \cite{XT2}, the authors extended this result, proving that symmetric biderivations of finite-dimensional simple Lie algebras are all trivial. Recently, significant attention has been devoted to studying biderivations of Lie and Lie super-algebras (see \cite{LLZ,XT,XT1,XT2,WYC,WY,MB,CYZ}). It is well-known that every biderivation can be expressed as the sum of a symmetric and a skew-symmetric biderivation. In \cite{BZ1}, the authors developed methods to investigate skew-symmetric biderivations for Lie algebras without centers, but specific approaches for symmetric biderivations remain limited. In \cite{WY}, it was shown that skew-symmetric biderivations of affine-Virasoro Lie algebra of type $A_1$ are inner.
In the present  short paper, we efficiently determine derivations as well as both symmetric and skew-symmetric biderivations of affine-Virasoro Lie algebras asociatied with any finite dimensional simple Lie algebra $\mf g$. As an application, we determine the commutatiassociateve post-Lie algebra structure on $ \mathfrak{L}(\mathfrak{g})$.

Post-Lie algebras represent a significant generalization of left-symmetric algebras, with important applications in algebra and geometry \cite{DB}. They have also been explored in contexts such as isospectral flows, Yang-Baxter equations, Lie-Butcher series, and moving frames \cite{FLMM}. Constructing non-trivial post-Lie algebra structures for a given pair of Lie algebras is challenging. In \cite{BD}, the authors introduced a specialized class of post-Lie algebra structures known as commutative post-Lie algebras. It was proven in \cite{BM} that any commutative post-Lie algebra structure on a complex, finite-dimensional perfect Lie algebra is trivial. Furthermore, in \cite{BZ}, it was shown that commutative post-Lie algebra structures on Kac-Moody Lie algebras are nearly trivial.

This paper is organized as follows. In section 2, we recall the definition of affine-Virasoro Lie algebra $\mf L(\mf g)$ associated with a finite dimensional simple Lie algebra $\mf g$. We  use the Hochschild-Serre spectral sequence to compute the derivation algebras of affine-Virasoro Lie algebras. It is shown that any derivation of affine-Virasoro Lie algebra is inner.
In section 3, we explicitly determine skew-symmetric and symmetric biderivations of affine-Virasoro Lie algebra $\mf L(\mf g)$, respectively.  
More precisely, it is shown that each skew-symmetric biderivation of $\mf L(\mf g)$ is inner, and each symmetric biderivation of $\mf L(\mf g)$ is trivial. Consequently, any biderivation of $\mf L(\mf g)$ is inner. In Section 4, as an application, we show that every commutative post-Lie algebra structure on $\mf L(\mf g)$ is trivial.	

\section{Derivations of affine-Virasoro Lie algebras}

Throughout this paper, $\Z$, $\C$, and $\C^*$ denote the sets of integers, complex numbers, and nonzero complex numbers, respectively. Moreover all   vector spaces, algebras, and tensor products are over $\C$.

\subsection{ Affine-Virasoro Lie algebras} Let $\mf g$ be a finite dimensional complex simple Lie algebra with the non-degenerate Killing form $(.)$. Let $\Phi$ denote the root system of $\mf g$. Let $\widehat {\mf g} = \mf g \ot \C[t^{\pm1}]$ be the loop algebra of $\mf g$ with the bracket operation defined by:
	$$ [x\otimes t^m,y\otimes t^n]=[x,y]\otimes t^{m+n},$$
	for all $x,y \in \mf g, m,n \in \Z.$ It is clear that $\what {\mf g}= \dis{\bgop_{n \in \Z}} \what{\mf g}_n,$ where $\what{\mf g}_n=\mf g \ot t^n$ for all $n \in \Z$.\\
	Now we define the affine-Viraosro Lie algebra $\mf L(\mf g)$. As a vector space: $$\mf L(\mf g)= \mf g\otimes  \C[t^{\pm1}]\oplus \C {K_1} \op \C {K_2} \op \; \dis{\bgop_{m \in \Z}} \C d_m,$$
	on which Lie brackets are given by:
	$$[x\otimes t^m,y\otimes t^n]=[x,y]\otimes t^{m+n}+m(x,y)\delta_{m+n,0}{K_1},$$
	$$[\mf L(\mf g), {K_1}]= [\mf L(\mf g), {K_2}]=0,$$
	$$ [d_m, x\otimes t^n]=nx\otimes t^{m+n},$$
	$$[d_m,d_n]=(n-m)d_{m+n}+\delta_{m+n,0}\frac{m^3-m}{12}{K_2},$$
	for all $x,y \in \mf g, m,n \in \Z.$ It is clear that $\mf L(\mf g)=\bigoplus_{n\in\Z}\mf L(\mf g)_n$ is $\Z$-graded with $\mf L(\mf g)_n=\dis\mf g \ot t^n \op \C d_n$, for all $n \neq 0$ and $\mf L(\mf g)_0=\mf g \op {\rm span}_{\C} \{  d_0, K_1,K_2 \}$, and $\mf L(\mf g)$ has center $Z(\mf L(\mf g))=\C K_1+\C K_2$.  Note that $\dis{\bgop_{m \in \Z}} \C d_m \op \C K_2$ is the classical Virasoro Lie algebra, we denote it by Vir.
	
	\subsection{Derivations of a Lie algebra} Let $G$ be a commutative group and $\mf G = \bgop_{g \in G} \mf G_g$ be a $G$-graded Lie algebra. A $\mf G$-module
	$V$ is called $G$-graded, if
	$V = \bgop_{g \in G}V_g$ such that
	$\mf G_g.V_h \subseteq V_{g+h} $ for all $g, h \in  G$.\\
	Let $\mf G$ be a Lie algebra and V be a $\mf G$-module. A linear map $D:\mf G  \to V$ is said to be a
	derivation from $\mf G$ to $V$, if $D$ satisfies the following property:
	$$D[x, y] = x . D(y) -y . D(x), \, \forall x,y \in \mf G.$$
	If there exists some $v \in V$ such that $D(x)= x.v$ for all $x \in \mf G$, then $D$ is called an inner derivation.
	Denote by ${\rm Der}(\mf G ,V)$ the vector space of all derivations and ${\rm Inn}(\mf G,V)$ the vector space of all inner derivations. Set
	$H^1(\mf G,V)= {\rm Der}(\mf G,V)/{\rm Inn}(\mf G, V)$.
	Let Der$(\mf G, \mf G)={\rm Der}(\mf G)$  denote the derivation algebra of $\mf G$ and  Inn$(\mf G, \mf G)={\rm Inn}(\mf G)$ be the vector space of all inner derivations of $\mf G$.\\
	
	In this section we determine Der$(\mf L(\mf g)).$  Note that $\wtil{\mf g}= \what{\mf g} \op \C K_1$ is an ideal of $\mf L(\mf g)$. Now consider the following short exact sequence of Lie algebras,
	$$ 0 \rightarrow \wtil{ \mf g} \rightarrow \mf{L}(\mf g) \rightarrow{}  {{\mf L(\mf g)}/{\wtil{\mf g}}} \rightarrow 0 , $$
	which induces an exact sequence
	\begin{align}\label{exact seq}
		H^1(\mf L(\mf g), \wtil { \mf g}) \rightarrow H^1(\mf L(\mf g), \mf L(\mf g)) \rightarrow H^1(\mf L(\mf g), {{\mf L(\mf g)}/{\wtil{\mf g}}}).
	\end{align}
	The right-hand side of the sequence (\ref{exact seq}) can be computed from the initial terms of the four-term sequence associated to the Hochschild-Serre spectral sequence (see \cite[\S 7.5]{WC}) given by:
	$$  0 \rightarrow H^1({{{\mf L(\mf g)}/{\wtil{\mf g}}}}, {{{\mf L(\mf g)}/{\wtil{\mf g}}}}) \rightarrow H^1(\mf L(\mf g), {{{\mf L(\mf g)}/{\wtil{\mf g}}}}) \rightarrow H^1(\wtil{ \mf g}, {{{\mf L(\mf g)}/{\wtil{\mf g}}}})^{{{\mf L(\mf g)}/{\wtil{\mf g}}}},$$
	where the term $H^1(\wtil{ \mf g}, {{{\mf L(\mf g)}/{\wtil{\mf g}}}})^{{{\mf L(\mf g)}/{\wtil{\mf g}}}}$ is equal to ${\rm Hom}_{U({{\mf L(\mf g)}/{\wtil{\mf g}}})}(\wtil{ \mf g}/[\wtil{ \mf g},\wtil{ \mf g}], {{{\mf L(\mf g)}/{\wtil{\mf g}}}})=0,$ since $\wtil{ \mf g}$ is a perfect Lie algebra. Moreover it is known from \cite[Theorem A.2.1]{ES} that $H^1({\rm Vir}, {\rm Vir})=0,$ and hence $H^1(\mf L(\mf g), {{{\mf L(\mf g)}/{\wtil{\mf g}}}})=0.$ Therefore, to determine $H^1(\mf L(\mf g), \mf L(\mf g)) $, we need to compute $ H^1(\mf L(\mf g), \wtil{ \mf g})$. By Proposition 1.1 of \cite{FR}, we have the following.
	\begin{lemma}
		Keep notations as before. Then ${\rm Der}(\mf L(\mf g)), \wtil{ \mf g})=\dis{\bgop_{n \in \Z}}{\rm Der} (\mf L(\mf g), \wtil{ \mf g})_n$,
		where $${\rm Der} (\mf L(\mf g),\wtil{ \mf g})_n= \{ \phi \in {\rm Der}(\mf L(\mf g), \wtil{ \mf g})\mid \phi(\mf L(\mf g)_m) \subseteq  \wtil{ \mf g}_{m+n}, \, \forall \, m \in \Z\}.$$
	\end{lemma}

We need the following two lemmas.

	\begin{lemma}\label{Lem 1 for der}
		For any $m \in \Z$ with $m \neq 0$, $H^1(\mf L(\mf g)_0, \wtil{\mf g}_m)=0$.
	\end{lemma}
	\begin{proof}
	 Let $m \neq 0$ and $\phi: \mf L(\mf g)_0 \to \wtil{\mf g}_m$ be a derivation. Since $H^1(\mf g, \wtil {\mf g}_m)=0,$ we have $\phi|_{\mf g}$ is an inner derivation. Hence there exits $v \in \wtil{\mf g}_m$ such that $\phi|_{{\mf g}}(X)=[X,v]$ for all $ X \in {\mf g}$. Now set $\wtil\phi=\phi-\phi_0$, where $\phi_0: \mf L(\mf g)_0 \to \wtil{\mf g}_m$ is defined by $\phi_0(X)=[X,v]$ for $X \in  \mf L(\mf g)_0$. In particular, we get $\wtil \phi|_{\mf g}=0$. Thus we can assume that $\phi|_{\mf g} =0.$ Now consider
	 $$0=\phi([y,x])=[\phi(y),x]+[y, \phi(x)]=[\phi(y),x], \, \, \forall \, x \in \mf g, \, y \in  \mf L(\mf g)_0.$$
	Consequently, $\phi( \mf L(\mf g)_0)=0$, as desired.

	\end{proof}
	\begin{lemma}\label{Lem 2 for der}
	For any $m ,n\in \Z$ with $m \neq n$, ${\rm Hom}_{\mf L(\mf g)_0}(\mf L(\mf g)_m, \wtil{\mf g}_n)=0$.
	\end{lemma}

	\begin{proof}
		Let $m ,n\in \Z$ with $m \neq n$, and $\phi \in {\rm Hom}_{\mf L(\mf g)_0}(\mf L(\mf g)_m, \wtil{\mf g}_n)$.

		Take any $X \in \mf L(\mf g)_m$ and let $\phi(X)= Y + a\de_{n,0}K_1,$ for some $Y \in {\mf g} \ot t^n$ and $a\in \C$. Since
	 \begin{align}
			m\phi(X)&=\phi([d_0,X])=[d_0,\phi(X)]=nY
		\end{align}
	It follows that $a=0$ when $n=0$, and hence we have $ (m-n)\phi(X)=0$, which yields that $\phi(X)=0$. This completes the proof.
	\end{proof}
	As a direct consequence of  Lemma \ref{Lem 1 for der}, Lemma \ref{Lem 2 for der} and \cite[Proposition 1.2]{FR}, we have
	\begin{proposition}
Keep notations as before, then we have
 $${\rm Der}(\mf L(\mf g), \wtil{ \mf g})={\rm Der}(\mf L(\mf g), \wtil{ \mf g})_0 + {\rm Inn}(\mf L(\mf g), \wtil{ \mf g}).$$
	\end{proposition}

We are now in a position to present the following result on derivations of $\mf L(\mf g)$.
	\begin{theorem}\label{Th derivation}
The following statements hold.
		\begin{itemize}
			\item[(1)] $H^1(\mf L(\mf g), \wtil{\mf g})=0$.
			\item[(2)] ${\rm Der}(\mf L(\mf g), \mf L(\mf g))={\rm Inn}(\mf L(\mf g)).$
		\end{itemize}
	\end{theorem}
	\begin{proof}
	By Proposition 2.4, to prove the statement (1), it is sufficient to show that Der$\mf L(\mf g), \wtil{ \mf g})_0\subseteq {\rm Inn}(\mf L(\mf g), \wtil{ \mf g})$. For that, let $$D \in {\rm Der}(\mf L(\mf g), \wtil{ \mf g})_0=\{ \phi \in {\rm Der}(\mf L(\mf g), \wtil{ \mf g})\mid\phi(\mf L(\mf g)_m) \subseteq  \wtil{ \mf g}_{m}, \, \forall \, m \in \Z\}.$$ Note that $D|_{\wtil {\mf g} }$ is a derivation of $\wtil{\mf g} $. It is well-known that Der$(\mf {\wtil g})= {\rm Inn}(\mf{\wtil g}) \op \C \ga,$ where $\ga:\mf {\wtil g} \to \mf {\wtil g}  $ is defined by $\ga(g)=[d_0,g]$ for all $ g \in \mf{\wtil g}$. Let $D|_{\mf {\wtil g}}=\ad X+c \ga  $, for some $c \in \C, X \in \mf{ \wtil g}$. Let $\phi:=D-\ad X-c\,\ad\, d_0$. Then it is clear that $\phi(Y)=0$ for all $Y \in \wtil{\mf g}.$ Hence, we can assume that $D|_{\wtil{\mf g} }=0$ without loss of generality. In the following, we always make this assumption. We have the following assertion.

{\bf Claim:} $D(d_m)=0$ for all $m \neq 0$.

Assume that $D(d_m)=X \ot t^m$ for some $X \in \mf g$. Since $D$ is a derivation, we have $$0=D([d_m,y])=[D(d_m),y]+[d_m,D(y)]
=[D(d_m),y], \, \, \forall\ \, y \in \mf g.$$
Therefore we have $[y, D(d_m) ]=[y,X]\ot t^m=0 $ for all $y \in \mf g$, i.e $X=0$ and hence the claim follows.\\
Now to prove $D(d_0)=0$, we consider the relation $2D(d_0)=D[d_{-1},d_1]=[D(d_{-1}),d_1]+[d_{-1},D(d_1)]=0,$ by above claim.\\
Now we need to show that $D(K_2)=0$. To prove this, consider the following relation $$D([d_m,d_{-m}])=[D(d_m),d_{-m}]+[d_m,D(d_{-m})]=0$$ for some $m \neq 1$, i.e, $D(2md_0+ \frac{m^3-m}{12}K_2) =0$. From this it follows that $D(K_2)=0$. This completes the proof of the statement (1).

Now the statement (2) follows directly from the exact sequence (\ref{exact seq}).
\end{proof}

\section{Biderivations of affine-Virasoro Lie algebras}
In this section, we determine biderivations of affine-Virasoro Lie algebras. We first recall the notion of biderivations.
	\begin{definition}
		Let $L$ be a Lie algebra and $V$ be an $L$-module. A bilinear map $F:  L \times  L \to  V$ is said to be a biderivation from $L$ to $V$ if it satisfies the following properties  for all $ \, x,y,z  \,  \in  L: $
		$$F([x,y],z)=x.F(y,z) - y.F(x,z) $$
		$$F(x,[y,z])=y.F(x,z)- z.F(x,y)  $$
		Further a biderivation from a Lie algebra $L$ to its adjoint representation is called a biderivation of $L$.
		
	\end{definition}
A biderivation $F$ is called symmetric (resp. skew-symmetric) if $F(x,y)=F(y,x)$ (resp. $F(x,y)=-F(y,x)$) for all $x,y \in L$. It is well-known that a biderivation of a Lie algebra $L$ can be expressed as a sum of symmetric biderivation and skew-symmetric biderivation. For any Lie algebra $L$ and $\la \in \C$, a natural biderivation $F_{\lambda}$ of $L$ is defined as  $F_\la(x,y)=\la [x,y]$ for any $x,y\in L$. We call biderivations of the form $F_\la$ as inner biderivations. In particular, for $\la=0$, $F_\la$ is the so-called trivial biderivation.

Note that for any biderivation $F$ of $ L$, we have $F(L, Z) \subseteq  Z$ and $F(Z, L) \subseteq  Z$, where $Z$ is the center of $L$, due to the following (\ref{center1}) and (\ref{center2}),
\begin{equation}\label{center1}
0=F(x,[y,z])=[F(x,y),z]+[y,F(x,z)] =[y,F(x,z)],  \, \forall \, x,y \in  L, \, z \in Z,
\end{equation}
and
\begin{equation}\label{center2}
0=F([x,y],z])=[x,F(y,z)]+[F(x,z),y] =[x,F(y,z)],  \, \forall \, x,z \in  L, \, y \in Z.
\end{equation}
It follows that it gives rise to an induced map $\bar F : L/Z \to L/Z$ defined by $\bar F(x+Z,y+Z)=F(x,y)+Z$. Clearly $\bar F$ is also a biderivation of $L/Z$. We have the following easy observation.

\begin{lemma}\label{biderivation with center}
Let $L$ be a perfect Lie algebra and $F$ be a biderivation of $L$. Then $F(L, Z)=F(Z, L)=0$.
\end{lemma}

\begin{proof}
The assertion follows direct from (\ref{center1}) and (\ref{center2}).
\end{proof}

%

The following result determines skew-biderivations of $\mf L(\mf g)/Z(\mf L(\mf g))$.

\begin{theorem}\label{skew-symmetric biderivation of L/Z}
	Every skew-symmetric biderivation of $\mf L(\mf g)/Z(\mf L(\mf g))$ is inner.
\end{theorem}
\begin{proof}
	Let $L=\mf L(\mf g)/Z(\mf L(\mf g))$, then $L=\bigoplus_{m\in\mathbb{Z}}L_m$ is $\mathbb{Z}$-graded. Let $F$ be a skew-symmetric biderivation of $L$. Then by\cite[Corollary 2.4]{BZ1}), there exists some $\ga \in \hom_L(L, L)$ such that $F(x,y)=\ga([x,y])$ for any $x,y \in L$.

Since $\ga$ is an $L$-module homomorphism, we get $\ga(L_m) \subseteq L_m$. Let $\ga(\overline{d_m})=\overline{X_m\ot t^m} +a_m\overline{d_m}$ for all $ m \in \Z$, $X_m \in \mf g, a_m \in \C$. Now we have
	$$0=\ga([\overline{y},\overline{d_m}])=[\overline{y},\overline{X_m\ot t^m}+a_m\overline{d_m}] =\overline{[y,X_m]\ot t^m}, \, \forall \, y \in \mf g.$$
	Hence $X_m=0$ for all $m \in \Z$. For any $m \neq 0$, we have $\ga(\overline{[d_{-m},d_m]})=[d_{-m},\ga(\overline{d_m})]$. This implies that $2m\ga(\overline{d_0})=2ma_m\overline{d_0}$, i.e., $2ma_0\overline{d_0}=2ma_m\overline{d_0}$. Hence $a_m=a_0$ for all $m \neq 0$. Moreover, for any $x\in\mf g$ and $m\neq 0$, since
$$-m\ga (\overline{x\otimes t^m})=\ga(\overline{[x\otimes t^m, d_0]})=[x\otimes t^m, \ga(\overline{d_0})]=-ma_0\overline{x\otimes t^m},$$
it follows that $\ga (\overline{x\otimes t^m})=a_0\overline{x\otimes t^m}$ for any $m\neq 0$. Since $\ga(L_0)\subseteq L_0$, we can assume
$\ga(\overline{x})=\overline{x^{\prime}}+c_x\overline{d_0}$ for $x, x^{\prime}\in\mf g, c_x\in\C$. Hence, for any $y\in\mf g$,
$$a_0\overline{[y,x]\otimes t}=\ga(\overline{[y,x]\otimes t})=\ga([y\otimes t, \overline{x}])=[y\otimes t, \ga (\overline{x})]=\overline{([y,x^{\prime}]-c_xy)\otimes t}.$$
Hence,
\begin{equation}\label{eq}
[y, x^{\prime}-a_0x]=c_xy, \,\forall\,y\in\mf g.
\end{equation}
In particular, if we take nonzero $y$ in a Cartan subalgebra of $\mf g$ in (\ref{eq}), we see that $c_x=0$. Then let $y$ be an arbitrary element of  $\mf g$ in (\ref{eq}), it yields that $x^{\prime}=a_0x$ for any $x\in\mf g$. Consequently,  $\ga = a_0 \, \text{Id}$. This completes the proof.
\end{proof}

The following result gives a connection between skew-derivations of  $\mf L(\mf g)$ and those of  $\mf L(\mf g)/Z(\mf L(\mf g))$.
\begin{proposition}\label{prop for L and L/Z}
The map $F \mapsto \bar F$ defines a bijection between skew-biderivations of $\mf L(\mf g)$ and skew-biderivations of $\mf L(\mf g)/Z(\mf L(\mf g))$.
\end{proposition}

\begin{proof}
It follows from Theorem \ref{skew-symmetric biderivation of L/Z} that any skew-biderivation of $\mf L(\mf g)/Z(\mf L(\mf g))$ is of the form $\overline{F_{\lambda}}$ for some $\lambda\in\C$, i.e., the map is surjective. Moreover, since $\mf L(\mf g)$ is  perfect, it follows from  \cite[Lemma 2.5] {BZ1} that the map is injective, so that it is bijective, as desired.
\end{proof}


As a direct consequence of Theorem \ref{skew-symmetric biderivation of L/Z} and Proposition \ref{prop for L and L/Z}, we have

\begin{corollary}\label{coro for skew}
	Every skew-symmetric biderivation of $\mf L(\mf g)$ is inner.
\end{corollary}

Before going to evaluate symmetric biderivations of $\mf L(\mf g)$, we recall some known results from \cite{LLZ}.
\begin{proposition}(cf. \cite[Proposition 3.1]{LLZ})\label{Prop for zero bider}
	Let $L$ be a finite-dimensional semisimple complex Lie algebra, and $V$, $W$ be $L$-modules with $V$ being finite-dimensional. Assume that a bilinear map $\delta: L \times V \to W$ satisfies the following conditions:
	\begin{itemize}
		\item[(1)] For each $x \in L$, $\delta(x, \cdot): V \to W$, $v \mapsto \delta(x, v)$ is a $L$-module homomorphism.
		\item[(2)] For each $v \in V$, $\delta(\cdot, v): L \to W$, $x \mapsto \delta(x, v)$ is a derivation from $L$ to $W$.
	\end{itemize}
	Then $\delta = 0$.
\end{proposition}
\begin{theorem}(cf. \cite[Theorem 3.2]{LLZ})\label{sym bider of semi simple}
	Let $L$ be a finite-dimensional semisimple complex Lie algebra, and $V$ be a finite-dimensional $L$-module. Then all symmetric biderivations $\delta: L \times L \to V$ are trivial.
\end{theorem}

We are now in a position to present the following result on symmetric biderivations of $\mf L(\mf g)$.

\begin{theorem}\label{thm for symmetric}
	Every symmetric biderivation of $\mf L(\mf g)$ is trivial.
\end{theorem}
\begin{proof}
Let $F$ be a symmetric biderivation of $\mf L(\mf g)$. We show that $F$ is trivial by several claims.

{\bf Claim 1:} $F(\mf g,\mf g)=0$.

 Note that $F(\mf g , \mf g)$ is a finite-dimensional vector space, and contained in $U(\mf g)(F(\mf g, \mf g))$ which is a finite dimensional $\mf g$ module, since the adjoint action of $\mf g$ on $\mf g$ is integrable and $\mf g$ acts on $Vir$ trivially. Thus $F|_{\mf g\times \mf g}: \mf g \times \mf g \to U(\mf g)(F(\mf g, \mf g))$ is a symmetric biderivation of $\mf g$. Hence Claim 1 follows from Theorem \ref{sym bider of semi simple}.
	
{\bf Claim 2:} $F(\mf g, \what{\mf g} )=0$.

Fix some $0 \neq m \in \Z$ and consider $F|_{\mf g , \,\mf g\ot t^m} :\mf g \times (\mf g\ot t^m) \to F(\mf g , \mf g\ot t^m)$. Note that $F(\mf g , \mf g\ot t^m) \subseteq U(\mf g)F(\mf g, \mf g\ot t^m)$, which is a finite-dimensional $\mf g$-module. Hence, to apply Proposition  \ref{Prop for zero bider}, we only need to show that $F(x,.):\mf g \ot t^m \to U(\mf g)F(\mf g , \mf g\ot t^m)$ is a $\mf g$-module homomorphism for any $x\in\mf g$, since $F$ is a biderivation. For all $y,g \in \mf g,$ we have
	\begin{align*}
		F(x,y.g \ot t^m) &=F(x,[y,g\ot t^m]) \cr
		&=[F(x,y),g\ot t^m]+[y,F(x,g\ot t^m)] \cr
		&=y.F(x,g\ot t^m).
		\end{align*}
Therefore we have $F(\mf g , \mf g\ot t^m)=0$. Since $m$ is arbitrary, this proves Claim 2.

{\bf Claim 3:} $F(\mf g, \mf L(\mf g))=0$.

Let ${\mf {sl}}_2^m={\rm {span}}_{\C}\{d_m,d_0,d_{-m}\}$ for all $m \geq 1$. Note that $\mf g$ acts trivially on $\mf{sl}_2^m$. Now consider $F|_{\mf g \times {\mf{sl}}_2^m}:\mf g \times {\mf{sl}}_2^m \to F(\mf g , {\mf{sl}}_2^m).$ In a similar manner as Claim 2, to apply Proposition \ref{Prop for zero bider}, we only need to show that $F(x, \cdot):\mf{sl}_2^m \to U(\mf g)F(\mf g , {\mf sl}_2^m)$ is a $\mf g$-module homomorphism for any $x\in\mf g$. Now for all $y \in \mf g, v \in \mf{sl}_2^m$, we have
\begin{align*}
	F(x,y.v)&=F(x,[y,v])=[F(x,y),v]+[y,F(x,v)]=y.F(x,v),
\end{align*}
where the last equality holds by Claim 2. It follows that $F(\mf g, d_m)=0$ for all $m \in \Z$ by Proposition \ref{Prop for zero bider}. This  together with Claim 2 and Lemma \ref{biderivation with center} proves  Claim 3.

{\bf Claim 4:}	$F(\what {\mf g}, \mf L(\mf g))=0$.

By \cite[Proposition 2.7(3)]{CYZ}, we have $F(\sigma(\mf g ), \mf L(\mf g))=0$ for all $\sigma \in \rm {Aut}(\mf L(\mf g))$. Let $\{x_\al \mid \al \in \Phi\}$ be the set of all root vectors of $\mf g$. Let $k=\min\{l\mid ({\rm{ad}}x_\al)^{l+1}=0, \,\forall\,\al \in \Phi\}$. Clearly $k+1 \geq 2$. Now consider  $\exp(\la\, {\rm {ad}}(x_\al \ot t^m)) \in \rm{Aut(\mf L(\mf g))}$ for $\al \in \Phi$, $m \neq 0$ and nonzero $\la \in \C$. Then we have
$$\exp(\la\, {\rm {ad}}(x_\al \ot t^m)) (\mf g)=\dis{\sum_{i=0}^{k}} \frac{\la^{i}}{i!}\underbrace{[x_\al,[x_\al,\cdots, [x_\al}_{i\,\rm times},\mf g ]\cdots]]\ot t^{im}.  $$

Now by choosing various $\la$, we can form a Vandermonde matrix and conclude that $[x_{\alpha},\mf g] \ot t^m \in \sum_{\sigma\in\rm {Aut}(\mf L(\mf g))}\sigma(\mf g)$. Since $\text{\rm span}_{\C}\{[x_\al,\mf g]\mid \al \in \Phi\}=\mf g$, we get $F(\mf g \ot t^m, \mf L(\mf g))=0$. As $m$ is arbitrary, the Claim follows.

{\bf Claim 5:} $F(\mf L(\mf g), \mf L(\mf g) )=0$.

Since $F(\wtil {\mf g}, \mf L(\mf g))=0$ by Claim 4 and Lemma \ref{biderivation with center}, we have  the following induced map
$$\wtil F: ({\mf L(\mf g)}/{\wtil{\mf g}}) \times  ({\mf L(\mf g)}/{\wtil{\mf g}}) \to  {\mf L(\mf g)}/{\wtil{\mf g}}$$ which is a symmetric biderivation of Vir. Now by \cite[Theorem 2.8]{TY}, we have $\wtil F=0$, i.e
\begin{equation}\label{an assertion}
F(\mf L(\mf g), \mf L(\mf g)) \subseteq \wtil{\mf g}.
\end{equation}
Hence, for any $m\in\mathbb{Z}, x,y\in\mf L(\mf g), z\in\mf g$,
$$0=F(x, [y, z\otimes t^m])=[F(x,y),z\otimes t^m]+[y, F(x, z\otimes t^m)]=[F(x,y),z\otimes t^m].$$
This together with (\ref{an assertion}) yields that $F(x, y)\in\C K_1$ for any $x, y\in\mf L(\mf g)$. Therefore,
$$F(x, [y,z])=[F(x, y), z]+[y, F(x, z)]=0,\,\forall\, x,y,z\in\mf L(\mf g).$$
Hence, the claim follows, as $\mf L(\mf g)$ is perfect. We complete the proof.
\end{proof}

As a direct consequence of Corollary \ref{coro for skew} and Theorem \ref{thm for symmetric}, we have
\begin{corollary}\label{coro for biderivation}
	Every  biderivation of $\mf L(\mf g)$ is inner.
\end{corollary}

	\section{An application}
An application of biderivations is the characterization of post-Lie algebra structures, which have been introduced by Valette in connection with the homology of partition posets and the study of Koszul operads (cf. \cite{VB}).
	\begin{definition}
		A commutative post-Lie algebra structure on a Lie algebra $L$ over a field $\mathbb{F}$ is an $\mathbb F$-bilinear product $x.y$ on $L$ which satisfy the following identities:
		$$x.y=y.x    $$
		$$  [x,y].z=x.(y.z)-y.(x.z)$$
		$$ x.[y,z]=[x.y,z]+[y,x.z], \, \forall \, x,y,z \in L. $$
In this sense, we call $(L,[,],.)$ as a commutative post-Lie algebra.
	\end{definition}

It is trivial to observe that every Lie algebra $L$ posses a trivial commutative post-Lie algebra structure with the operation $x.y=0$ for all $x,y \in L.$ Moreover, it is a challenging question to determine when a Lie algebra admits a non-trivial commutative post-Lie algebra structure and find all non-trivial commutative post-Lie algebra structures. There is a connection between commutative post-Lie algebra structures over a Lie algebra $L$ and symmetric biderivations on $L$, given by the following.
	\begin{lemma}[cf. \cite{XT1}]\label{connection}
		Let $(L,[,],.)$ be a commutative post-Lie algebra over $L.$ Then the bilinear map $\de: L \times L \to L$ defined by $\de(x,y)=x.y$ is a symmetric biderivation on $L$.
	\end{lemma}
	Owing to Lemma \ref{connection} and Corollary \ref{coro for biderivation}, we have the following result on commutative post-Lie algebra structure on $\mf L(\mf g)$.
	\begin{theorem}\label{Thm comm post str}
		Every commutative post-Lie algebra structure on $\mf L(\mf g)$ is trivial.	
	\end{theorem}

	\vspace{1cm}
	{\bf Acknowledgments:}
	The first and third authors would like to thank Shanghai Maritime University, Shanghai, China for the hospitality during their visit period, where most of the work of this project was done. Further the first author would like to thank prof. Hengyun Yang for her generous guidance in his visit period.  Y.Y. is partially supported by the National Natural Science Foundation of China (12271345 and 12071136).
K.Z. is partially supported by  NSERC (311907-2020).


\begin{thebibliography}{100}
		\bibitem[B]{B} D. Benkovic, Biderivations of triangular algebras, Linear Algebra Appl. 431, 1587-1602 (2009).
\bibitem[MB]{MB} M. Bresar, On generalized biderivations and related maps, J. Algebra 172, 764-786 (1995).
		\bibitem[MB1]{MB1} M. Bresar, Near-derivations in Lie algebras, J. Algebra 320, 3765-3772 (2008).
\bibitem[BMM]{BMM} M. Bresar, W. S. Martindale, C. R. Miers, Centralizing maps in prime rings with involution, J. Algebra
		161, 342-357 (1993).
\bibitem[BZ1]{BZ1} M. Bresar, K.Zhao, Biderivations and commuting linear maps on Lie algebras, J. Lie Theory 28,
		885-900 (2018).	
\bibitem[DB]{DB} D. Burde, Left-symmetric algebras, or pre-Lie algebras in geometry and physics, Cent. Eur. J. Math.
		4(3), 323-357 (2006).
		\bibitem[BD]{BD} D. Burde, K. Dekimpe, Post-Lie algebra structures on pairs of Lie algebras, J. Algebra 464, 226-245
		(2016).
		\bibitem[BM]{BM} D. Burde, W. A. Moens, Commutative post-Lie algebra structures on Lie algebras, J. Algebra 467,
		183-201, (2016).
		\bibitem[BZ]{BZ} D. Burde, P. Zusmanovich, Commutative post-Lie algebra structures on Kac-Moody algebras. Comm. Algebra 47(12), 5218-5226, (2019).
		\bibitem[CYZ]{CYZ} Q. Chen, Y. Yao, K. Zhao, Biderivations of Lie algebras, Can. Math. Bull. 68(2), 440-450 (2025).
	
		\bibitem[ES]{ES} J. Ecker, M. Schlichenmaier, The Vanishing of the low dimensional cohomology of the Witt and the Virasoro algebra, http://arxiv.org/abs/1707.06106v2.
		\bibitem[FR]{FR} R. Farnsteiner,  Derivations and extensions of finitely generated graded Lie algebras, J. Algebra, 118, 34-45 (1988).
		\bibitem[FLMM]{FLMM}  K. Ebrahimi-Fard, A. Lundervold, I. Mencattini, H. Munthe-Kaas, Post-Lie algebras and isospectral
		flows, SIGMA Symmetry Integrability Geom. Methods Appl. 11, 093 (2015).
		\bibitem[LLZ]{LLZ} S. Liu, D. Liu, Y. Zhao, Symmetric biderivations on complex semisimple Lie algebras, to appear in  Algebra Colloq., 	arXiv:2407.05581.	
	\bibitem[XT]{XT} X. Tang, Biderivations and commutative post-Lie algebra structures on the Lie algebra $W(a,b)$, Taiwanese J. Math. 22(6), 1347-1366 (2018).
		\bibitem[XT1]{XT1} X. Tang, Biderivations, commuting maps and commutative post-Lie algebra structures on W-algebras,
		Commun. Algebra 45(12), 5252-5261 (2017).
		\bibitem[XT2]{XT2}  X. Tang, Biderivations of finite-dimensional complex simple Lie algebras, Linear Multilinear A. 66(2),
		250-259 (2018).
		\bibitem[TY]{TY} X. Tang, Y. Yang, Biderivations of the higher rank Witt algebra without anti-symmetric condition,
		Open Math. 16, 447-452 (2018).
	
		\bibitem[WY]{WY} Y. Wu, X. Yue,   Skew-symmetric biderivation on the affine-Virasoro algebra of type $A_1$ and their applications, Comm. Algebra, Vol. 50(2), 770-783 (2022).
		\bibitem[WYC]{WYC}  D. Wang, X. Yu, Z. Chen, Biderivations of parabolic subalgebras of simple Lie algebras, Commun.
		Algebra 39(11), 4097-4104 (2011).
	\bibitem[WC]{WC} C. Weibel, An Introduction to Homological Algebra. Cambridge Studies in
		Advanced Mathematics, (1994) Cambridge University Press.
		

\bibitem[VB]{VB} B. Vallette,  Homology of generalized partition posets. J. Pure Appl. Algebra,  208(2), 699-725 (2007).	
	\end{thebibliography}
\end{document}